\documentclass{article}

\usepackage{arxiv}

\usepackage{amsthm}
\usepackage{hyperref}
\usepackage{url}
\usepackage{hyperref}
\usepackage{url}
\usepackage[utf8]{inputenc} 
\usepackage[T1]{fontenc}    
\usepackage{booktabs}       
\usepackage{amsfonts}       
\usepackage{nicefrac}       
\usepackage{microtype}      
\usepackage{mathtools}
\usepackage{amssymb}
\usepackage[shortlabels]{enumitem}
\usepackage{xcolor,colortbl}
\usepackage{comment}
\usepackage{subcaption}
\usepackage{graphicx}
\usepackage{makecell}
\usepackage{multirow}
\usepackage{float}
\usepackage{booktabs,caption}
\usepackage[flushleft]{threeparttable}
\usepackage{ulem}
\usepackage{algorithm}
\usepackage{algpseudocode}
\usepackage[T1]{fontenc}
\usepackage{lmodern}

\newcommand{\Var}{\operatorname{Var}}

\newcommand{\req}[1]{Eq.\,(\ref{#1})}

\newtheorem{theorem}{Theorem}[section]
\newtheorem{lemma}[theorem]{Lemma}
\newtheorem{proposition}{Proposition}
\newtheorem{corollary}{Corollary}

\theoremstyle{definition}

\theoremstyle{remark}
\newtheorem{remark}[theorem]{Remark}

\numberwithin{equation}{section}

 \title{Concentration Inequalities and UQ Bounds for  Hypocoercive MCMC Samplers}

\author{Jeremiah Birrell\\Department of Mathematics\\ Texas State University, 601 University Drive, San Marcos, TX, 78666, USA\\
\texttt{jbirrell@txstate.edu}\\
\And
Luc Rey-Bellet\\
Department of Mathematics and Statistics\\
University of Massachusetts Amherst\\
710 N Pleasant St, Amherst, MA, 01002, USA\\
\texttt{luc@math.umass.edu}\\
\And
Dedicated to Professor Paul Dupuis on the occasion of his 65th birthday.
}

\begin{document}

\maketitle

\begin{abstract}
In this work we provide performance guarantees for hypocoercive non-reversible MCMC samplers $X_t$ with invariant measure       $\mu_*$; our results apply in particular to the Langevin equation, Hamiltonian Monte-Carlo, and the bouncy particle and  zig-zag samplers. Specifically, we establish a concentration inequality of Bernstein type for ergodic averages $\frac{1}{T} \int_0^T f(X_t)\, dt$. As a consequence we provide two types of performance guarantees:  (a) explicit non-asymptotic confidence intervals  for $\int f d\mu_*$ when using a finite time ergodic average with given initial condition $\mu$ and (b) uncertainty quantification (UQ) bounds, expressed in terms of relative entropy rate, on the bias of $\int f d\mu_*$ when using  an alternative or approximate processes $\widetilde{X}_t$. (Results in (b) generalize  results from \cite{BRBMarkovUQ} for coercive dynamics.) The concentration inequality is proved by combining the approach via Feynman-Kac semigroups first noted by \cite{Wu} with the hypocoercive estimates of \cite{Dolbeault2009,Dolbeault2015} developed for the Langevin equation and  generalized to partially deterministic Markov processes by  \cite{ADNR2018}.   
\end{abstract}


\section{Introduction and statement of the results}\label{introduction}

 Consider the problem of computing the  expected value 
\begin{align}
\nu_*[f]\equiv \int f(q)  d\nu_*(q)  \textrm{ with }   d \nu_* =  Z^{-1} e^{- \beta V(q)} dq
\end{align}
for some given function $f: \mathcal{Q} \to \mathbb{R}$ with $\mathcal{Q}\subset \mathbb{R}^d$. If the normalization constant $Z$ is unknown or prohibitive to compute, it can be advantageous to construct a ergodic stochastic process $Q_t$ with stationary distribution $\nu_*$ (in this paper, only  continuous-time processes are considered) and use the fact that, by the strong law of large numbers and for suitable $f$, one has
\begin{align}
\lim_{T \to \infty} \frac{1}{T} \int_0^T f(Q_t) \, dt  = \nu_*[f]\,.
\end{align} 
Such a process $Q_t$ is usually called a Monte-Carlo Markov chain (MCMC) and 
we can then use the finite time average $\frac{1}{T} \int_0^T f(Q_t)dt$ as an estimator for $\nu_*[f]$. 
There are of course multiple choices of stochastic processes with invariant 
measure $\nu_*$ and in order to decide which process  to use  we need to evaluate its performance.

While traditional Monte-Carlo algorithms are often built to be reversible, in recent years  non-reversible algorithms have attracted a lot of attention because of their potential to sample the space in a more efficient manner, in particular in the context of Bayesian statistics and molecular dynamics   (see for example \cite{DHN,HwHwSh2005,Bi2016,DLP2016,RBS3} and many more references therein).
In this paper we consider a variety of non-reversible MCMC samplers such as the Langevin equation, and various modifications thereof, as well as partially deterministic Markov processes such as the zig-zag sampler (\cite{BFR2019}),  the bouncy particle sampler (\cite{PetersdeWith2012}), and the hybrid Hamiltonian Monte-Carlo (\cite{Duane1987}).  
Each of these samplers are  constructed by extending the phase space from $\mathcal{Q}$ to $\mathcal{X} =  \mathcal{Q} \times \mathcal{P}$ and then constructing a non-reversible MCMC in the extended phase space with an invariant measure $\mu_*=\nu_* \times \rho_*$. 
The extra dimension  $\mathcal{P}$ can be thought as  momentum space and the dynamics considered here combine a conservative Hamiltonian-type dynamics with a dissipative sampling mechanism for the measure $\rho_*$ in the momentum variable $p \in \mathcal{P}$. For example $\rho_*$ may be a Gaussian distribution, although other choices are possible; we do assume $\rho_*$ has mean zero and nonzero covariance.

All of the algorithms we consider here have been proved to be hypocoercive. 
The concept of hypocoercivity was formalized by Villani  to describe dynamics 
which do not satisfy  a Poincar\'e inequality (otherwise they would be called coercive) but yet converge exponentially fast to equilibrium in $L^2(\mu_*)$.  In a series of work \cite{Desvillettes2001,HerauNier2004,EckmannHairer2003,Herau2006,villani2009hypocoercivity} it was proved that the Langevin equation is hypocoercive (see also \cite{EPR1998,EckmannHairer2000,RBT2002,Mattingly2002} for some earlier and related convergence results). 
   More recently \cite{Dolbeault2009,Dolbeault2015} found a new, short and very elegant, proof of hypocoercivity, and their techniques have been used for various modifications of the Langevin equation \cite{Iacobucci2017,StoltzTrs2018,StoltzVdE} 
   (some of them without hypoellipticity) 
   and    for a class of partially deterministic Markov  processes \cite{ADNR2018}, among them:
   \begin{enumerate}
   \item The bouncy particle sampler, which was introduced  in \cite{PetersdeWith2012} and whose ergodic properties were studied in \cite{BVC2018} and \cite{WR2017}. 
   \item The zig-zag sampler, introduced in \cite{BFR2019} and further studied 
   in \cite{BRZ2017}, which generalize to higher dimension the so-called telegraph process studied earlier in \cite{FGM2012,FGM2016} and \cite{MonMarch2014}. 
   \item The hybrid Hamiltonian Monte Carlo introduced by \cite{Duane1987} and whose ergodic properties are studied in \cite{BRSS2017}, see also  \cite{ELi2008} and  \cite{Neal2011}. 
   \end{enumerate}
   
 To explain this result,    decompose the generator $A$ of the dynamics on the Hilbert space $L^2(\mu_*)$ into symmetric and antisymmetric parts,  
$A= S+T$ with $S^*=S$, $T^*=-T$, denote by $\Pi$ the projection of $L^2(\mu_*)$ onto $L^2(\nu_*)$ given $\Pi(f)(q) = \int f(q,p) d \rho_*(p)$, and consider the operator 
\begin{align}
B= (I +  (T \Pi)^* (T \Pi))^{-1} (- T \Pi)^*  
\end{align}
and the family of modified scalar products $\langle  f,\, g\rangle_\epsilon = \langle  f,\,(I + \epsilon G) g\rangle$, where $G=B+B^*$; here, and in the following, $\langle \cdot,\cdot\rangle$ will denote the $L^2(\mu_*)$-scalar product. Under suitable conditions (more details are in Section \ref{sec:hypocoercive})
the norms induced by these modified scalar products are equivalent to the $L^2(\mu_*)$-norm and it is shown in \cite{Dolbeault2009,Dolbeault2015,ADNR2018} that, for sufficiently small values of $\epsilon>0$, the dynamics satisfy a Poincar\'e inequality for the modified scalar product
    \begin{equation}\label{eq:hypocoercive-epsilon}
    \langle - A f , f \rangle_\epsilon \ge \Lambda(\epsilon) \, {\rm Var}_{\mu_*}(f) \,,	
    \end{equation}
where $\Lambda(\epsilon)>0$ can be explicitly bounded in terms of  the Poincar\'e constant of the measure $\nu_*$, the spectral gap of the sampling dynamics for $\rho_*$, and properties of the potential $V$ (see Section \ref{sec:hypocoercive}).  

%
%

In this paper we leverage this approach to hypocoercivity to prove concentration inequalities (of Bernstein type) for finite time ergodic averages of a function (observable) $f:\mathcal{X}\to\mathbb{R}$:
\begin{align}
F_T=\frac{1}{T}\int_0^T f(X_t)dt\,.
\end{align}
For reversible processes, or more generally for processes whose reversible part satisfies a Poincar\'e inequality, that is for {\em coercive} processes,  concentration inequalities were obtained  first in \cite{lezaud:hal-00940906}   and then both simplified and greatly generalized in \cite{Wu,cattiaux_guillin_2008,Guillin2009,GaoGuillinWu}; our approach relies heavily on the ideas developed in those works.

Concentration inequalities are very useful for providing performance guarantees (e.g., confidence intervals) valid for all time $T$ (and which do not rely on the central limit theorem).  In practice, algorithm performance is typically evaluated in terms of the asymptotic variance 
\begin{align}
\sigma^2(f) \equiv \lim_{T \to \infty} T \Var \left( \frac{1}{T}\int_0^T f(X_t) \, dt\right);
\end{align} we take here the alternative point of view of using  concentration inequalities, thereby obtaining {\em non-asymptotic} performance guarantees.  In a related manner, it has been advocated in \cite{DupuisSwapping1} that, to evaluate the performance of a MCMC algorithm,  one consider the large deviation rate for the empirical measure itself, an approach used in \cite{RBS1} to analyze  non-reversible perturbations of the overdamped Langevin equation. 

We prove the following non-asymptotic performance guarantee in Section \ref{sec:hypocoercive} (see Corollary \ref{cor:confidence-interval}). Here, and in the following, $(X_t,P^\mu)$ will denote a $\mathcal{X}$-valued Markov process with initial distribution $X_0 \sim \mu$ (i.e., $\mu=(X_0)_*P^\mu$ is the pushforward of $P^\mu$ by $X_0$), $E^\mu$ will be the expectation with respect to $P^\mu$, and  $\|\cdot\|$ will denote the $L^2(\mu_*)$-norm.

\begin{theorem}{\bf(Non-asymptotic confidence intervals)}\label{thm:hypo-conf}
Suppose that the Markov process $(X_t,P^\mu)$ satisfies the hypocoercive estimate \eqref{eq:hypocoercive-epsilon}. Then for any bounded observable $f$, any time $T>0$, and tolerance level $0< 1-\delta < 1$  
we have 
\begin{align}
P^\mu\left( \left| \frac{1}{T}\int_0^Tf(X_t)dt-\mu_*[f]\right| \le r \right)\geq 1-\delta\,,
\end{align}
where    
\begin{align}
r =\sqrt{2 v \frac{1}{T}\log\left(\frac{2N}{\delta}\right)} + b
\frac{1}{T}\log\left(\frac{2N}{\delta}\right)\,,
\end{align}
with 
\begin{align} \label{eq:vbM}
v=\frac{ (1+\epsilon) (1 - \frac{\epsilon^2}{4})}{1-\epsilon} 
\frac{2 \Var_{\mu_*}[f]}{\Lambda(\epsilon)}\,, \,  b =  \frac{(1+\epsilon)^2}{1-\epsilon}\frac{\|\widehat{f}\|_\infty}{\Lambda(\epsilon)}\,,\, N = \frac{\left\|\frac{d\mu}{d\mu_*}\right\| }{\sqrt{1-\epsilon}}\,.
\end{align}
$\Lambda(\epsilon)$ is defined in \req{Lambda_m_def} and $\epsilon\in(0,1)$ must satisfy \req{eq:Lambda_pos}.
\end{theorem}
 We emphasize that the novelty of Theorem \ref{thm:hypo-conf} lies in our technique for obtaining concentration inequalities from  a Poincar\'e inequality  in a modified inner product, thereby extending  established non-asymptotic results for coercive systems  \cite{lezaud:hal-00940906,Wu,cattiaux_guillin_2008,Guillin2009,GaoGuillinWu} to the hypocoercive setting for the first time. As our derivation fundamentally relies on the convergence rate obtained via the hypocoercivity method of \cite{Dolbeault2009,Dolbeault2015}, we note that Theorem \ref{thm:hypo-conf} does not provide substantial new insight regarding performance comparisons between coercive and hypocoercive samplers, beyond what is already apparent from the (modified)  Poincar\'e inequalities. For the derivation of explicit constants that are the input to our results,  we refer the reader to, e.g., \cite{ADNR2018}.

In addition to concentration inequalities, we also prove a robustness result for the dynamics, with respect to model-form uncertainty. For such uncertainty quantification (UQ) bounds, we think of $(X_t,P^\mu)$, called the baseline model, as an imperfect representation of a ``true" (or at least, more precise) alternative model. This alternative model may not be fully known, or it might be intractable (analytically or numerically), and so one may want to investigate how sensitive the results for the baseline model are to (not necessarily small) model perturbations. The next theorem provides such performance guarantees, generalizing the results in \cite{BRBMarkovUQ}, and is based on the general approach to uncertainty quantification introduced in \cite{chowdhary_dupuis_2013} and further developed in \cite{DKPP,KRW,GKRW,birrell2021quantification}.  In this context, the goal is to control the bias 
\[
\widetilde{E}^{\widetilde{\mu}}\left[ \frac{1}{T}\int_0^Tf(\widetilde{X}_t)\, dt \right] - \mu_*[f]\,,
\]
where  $(\widetilde{X}_t,\widetilde{P}^{\widetilde{\mu}})$ with $\widetilde X_0 \sim \widetilde{\mu}$ 
is the alternative model and $\widetilde{E}^{\widetilde{\mu}}$ is the expectation with respect to $\widetilde{P}^{\widetilde{\mu}}$.   We denote by $P^\mu_T$ and $\widetilde{P}_T^{\widetilde{\mu}}$  the path-space  distributions of the base and alternative models on the time window $[0,T]$ and  prove the following result in Section \ref{sec:concentration} (see Theorem \ref{thm:UQ_mod_poincare}). 
%
%
%
%
%
\begin{theorem}{\bf(Uncertainty quantification bounds)}\label{thm:hypo-uq}  
Suppose that the baseline Markov process $X_t$ satisfies the hypocoercive estimate \eqref{eq:hypocoercive-epsilon} and $(\widetilde X_t,\widetilde{P}^{\widetilde{\mu}})$ is a stochastic process such that the path space relative entropy satisfies
\begin{align}
R \big( \widetilde{P}_T^{\widetilde{\mu}}\|P_T^{\mu_*}\big) <\infty\,.
\end{align}	
Then for any bounded observable $f$ and any time $T>0$, we have    
\begin{align}  
	\left| \widetilde{E}^{\widetilde\mu} \left[ \frac{1}{T}\int_0^Tf(\widetilde X_t) \, dt   \right] - \mu_*[f] \right|
\leq &  \sqrt{ 2 v \eta_T}+  b\eta_T \,,
\end{align}  
where $v$ and $b$ are given in \eqref{eq:vbM},  and 
\begin{align}
\eta_T =\frac{1}{T} \left(\log\left((1-\epsilon)^{-1/2}\right)+  R\big(\widetilde P^{\widetilde{\mu}}_T\|P^{\mu_*}_T\big) \right)\,.\notag
\end{align}
\end{theorem}

Note that if the perturbed dynamics $(\widetilde X_t,\widetilde{P}^{\widetilde{\mu}})$ is itself ergodic (Markovian or not),  one can often prove that the entropy production rate $\lim_{T \to \infty}\frac{1}{T}R\big(\widetilde P^{\widetilde{\mu}}_T\|P^{\mu_*}_T\big) $ exists and is finite. In that case, Theorem \ref{thm:hypo-uq}
provides an uncertainty quantification bound for the expectation of $f$ under corresponding stationary distributions.

The remainder of this paper is organized as follows.  In Section \ref{sec:hypocoercive}
we review several examples of hypocoercive systems to which our results apply.  There, we also give an overview of the hypocoercivity method of \cite{Dolbeault2009,Dolbeault2015}.  This method is a crucial tool in the proofs of our new results, namely the concentration inequalities and UQ bounds outlined above; proofs of these are given in Section \ref{sec:concentration}.

\section{Hypocoercive MCMC samplers}\label{sec:hypocoercive}  
 
In this section we introduce several examples of popular hypocoercive samplers for which the modified Poincar\'e
inequality \eqref{eq:hypocoercive-epsilon}  has been proven by following the strategy of \cite{Dolbeault2009,Dolbeault2015}. 
In particular we consider several examples of partially deterministic MCMC samplers studied in   \cite{ADNR2018}. 
 We will refer the  reader to the original papers for technical details  and content ourselves with a brief, and at times somewhat informal, overview:

Consider a probability measure $d \nu_*(q)=Z^{-1} e^{-\beta V(q)}dq$ on 
$\mathbb{R}^d$ to be sampled, and for which a Poincar\'e inequality holds, i.e., there exists a constant $C_{\nu_*}>0$ such that for all for $g \in L^2(\nu_*)$ we have  
\begin{align}
\| \nabla_q g \|^2_{L^2(\nu_*)} \ge C_{\nu_*} \Var_{\nu_*}[g]\,.	
\end{align}
See, e.g., \cite{bakry2008} for conditions on $V$ which imply   a Poincar\'e inequality. 

Define the product measure $\mu_*= \nu_*\times\rho_*$ 
on the extended phase space $\mathbb{R}^d \times \mathcal{P}$   and  the projection $\Pi f  = \int f d\rho_*$.  
We consider a Markov processes $X_t=(Q_t,P_t)$ on  $\mathbb{R}^d\times \mathcal{P}$ 
with invariant measure $\mu_*$ and assume standard smoothness and growth conditions on $V$ to ensure that $X_t$  induces a strongly continuous semigroup $T_t$ on $L^2(\mu_*)$ with generator $A$, and with the time-reversed process having generator  given by the adjoint $A^*$ of $A$ on $L^2(\mu_*)$.  We decompose $A$ into symmetric and  antisymmetric parts: 
\begin{equation}
A = S+T,   \textrm{  with }   S= \frac{A + A^*}{2}  \textrm{ and }     T= \frac{A - A^*}{2}\,. 
\end{equation}

The following  four examples  fit within this framework and that will be used to illustrate the utility of our results; see \cite{ADNR2018}  for a proof of hypocoercivity of a more general class of models which covers all examples considered here,  
as well as \cite{StoltzTrs2018,Iacobucci2017,StoltzVdE} for further examples (some of them being non-equilibrium as well). 

\begin{enumerate}
\item{\bf (Langevin and modified Langevin equations)} The (underdamped) Langevin equation is the system of stochastic differential equations on $\mathbb{R}^{2d}$ given by 
\begin{align}
&dQ_t=\frac{P_t}{m}dt,\,\,\,\,\,
dP_t=\left(-\nabla V(Q_t)-\gamma \frac{P_t}{m}\right)dt+\sqrt{\frac{2\gamma}{\beta}} dW_t,
\end{align}
where $m>0$ is the mass, $\beta>0$ is proportional to the inverse temperature, $\gamma>0$ is the drag coefficient,  $W_t$ is a Wiener process, and $V:\mathbb{R}^d\to\mathbb{R}$ is a smooth potential. The appropriate  $\rho_*$ is a Gaussian measure with mean $0$ and covariance matrix $m/\beta I$.  
The generator $A$ is an extension of the differential operator 
\begin{equation}
A = 	\underbrace{ \frac{\gamma}{\beta}  \Delta_p - \gamma \left(\frac{p}{m}\right)^T \nabla_p}_{=S} + \underbrace{ \left( \frac{p}{m}\right)^T \nabla_q-\nabla V(q)^T\nabla_p}_{=T}\,.
\end{equation}
This is the model originally considered in \cite{Dolbeault2009,Dolbeault2015}
and several  modifications of this models have also been shown to be hypocoercive. For example \cite{StoltzTrs2016,StoltzTrs2018} consider a Langevin equation with a modified kinetic energy (non-quadratic) so that that $\rho_*$ is not Gaussian and the diffusion needs not be hypoelliptic.  Further generalizations of the Langevin equations with general $\rho_*$ are also considered in \cite{ADNR2018}.

\item{\bf (Hybrid Hamiltonian Monte Carlo)}  In this randomized version of Hamiltonian Monte-Carlo  introduced by \cite{Duane1987}, 
 the system follows Hamiltonian equations  of motion  with Hamiltonian $V(q) + p^2/2m$ for an  exponentially distributed amount of time, after which the momentum is resampled from the Gaussian measure $\rho_*$.  The generator has the form 
\begin{align} 
A  &=  \underbrace{ \lambda(\Pi - I )}_{=S} + \underbrace{ \left( \frac{p}{m}\right)^T \nabla_q-\nabla V^T\nabla_p}_{=T}\,.
\end{align}

\item{\bf (Bouncy Particle Sampler)}  In this sampler, introduced originally in \cite{PetersdeWith2012}, a particle starting at time $t_0$ in the state $(q_0,p_0)$ moves freely  $p(t)=p(t_0)$ and $q(t) =q(t_0) + t \frac{p(t_0)}{m}$  up to the random time $t_0 + \tau$. The updating time $\tau$ is governed by two mechanisms: either the velocity of the particle is refreshed,  i.e.,  $p$ is sampled from the Gaussian $\rho_*$ (this occurs at rate $\lambda$), or the particle ``bounces", i.e., it undergoes  a Newtonian elastic collision on the hyperplane tangential to the gradient of the energy and  the momentum is updated according to 
the rule 
\begin{equation}
R(q) p =  p  - \frac{  p^T \nabla V(q)    }{ \| \nabla V\|^2}  \nabla V \,.
 \end{equation}
The time at which this happens is   governed by an inhomogeneous Poisson process of intensity 
$\lambda(q,p) =   \left[\left( \frac{p}{m}\right)^T  \nabla V(q)\right]^+$. 
If we set $R f( q,p)  =  f(q, R(q) p)$  then the generator is 
\begin{align}
	A & =  \left(\frac{p}{m}\right)^T \nabla_q    +    \left[\left( \frac{p}{m}\right)^T  \nabla V(q)\right]^+  (R - I)   + \lambda (\Pi - I) \,,
	\end{align}
and elementary computations shows that $\mu_*=\nu_*\times\rho_*$ is invariant 
and  
%
\begin{align} 
S &=   \,    \left|\left( \frac{p}{m}\right)^T  \nabla V(q)\right|  (R - I)  +  \lambda(  \Pi - I )  \,, \\
T&=   \left( \frac{p}{m}\right)^T \nabla_q  +    \left( \frac{p}{m} \right)^T  \nabla V(q) (R - I)   \,.
\end{align}

\item{\bf (Zig-Zag Sampler)}  In the zig-zag sampler, contrary to the other examples, the velocity is discrete, and, for example, $\rho_*$ is the uniform 
distribution on $\{-1,1\}^d$.  As in the bouncy sampler, the trajectories are piecewise linear.  At updating  times,  the (randomly chosen) $i$'th component 
of the velocity is reversed; see  \cite{BFR2019} for a more detailed discussion. The generator of the Markov process has the form 
\begin{equation}
A =   v^T \nabla_q    +    \sum_{i=1}^d  \left[  v_i \partial_{q_i} V(q)\right]^+  (R_i - I)   + \lambda (\Pi - I)  \,,
\end{equation}
where $R_i f(q,v) = f(q, v - 2 (e_i^T v) e_i)$ (with $e_i$ the standard basis vector in $\mathbb{R}^d$). A computation similar to the one for the bouncy sampler shows that 
\begin{align} 
S &=   \,    \sum_{i=1}^d \left|  v_i   \partial_{q_i} V (q) \right|  (R_i - I)  +  \lambda(  \Pi - I )  \,, \\
T&=   v \nabla_q  +    \sum_{i=1}^d   v_i \partial_{q_i} V(q)  (R_i - I)   \,.
\end{align}
\end{enumerate}

Note  that for all the examples considered, it is easy to verify that 
one has the identity 
\begin{equation}\label{eq:TPi}
T \Pi = \frac{p}{m} \nabla_q \Pi
\end{equation}
(with the convention that $p/m=v$ for the zig-zag sampler). 
This fact is used to establish the following functional analytic estimates (see \cite{Dolbeault2009,Dolbeault2015})   which are the basis for the hypocoercive estimates (for the convenience of the reader the proof is in Appendix \ref{app:additional_proofs}).

%
%

\begin{proposition}\label{thm:elem}
Define
\begin{equation}
B= (I +  (T \Pi)^* (T \Pi))^{-1} (- T \Pi)^* \,. 
\end{equation}
The operators $S$,  $T$, and $B$ have the following properties:
\begin{enumerate}
\item $B1 = B^*1 = 0$,
\item   $S = (I -\Pi) S (I- \Pi)$,
\item   $T \Pi = (I- \Pi) T \Pi$,
\item   $B = \Pi B = \Pi B (I - \Pi)$  and $B$ and $TB$ are bounded operators with $\|Bf \| \le 1/2 \| (I - \Pi) f\|   $ and $\|T B f \| \le  \| (I - \Pi) f\|$. 
\end{enumerate}
\end{proposition}
 
 Next, define the family of modified scalar products on $L^2(\mu_*)$,  
 \begin{align}\label{mod_norm_def}
 \langle f , g\rangle_{\epsilon}=\langle f, g\rangle + \epsilon\langle f ,(B+B^*) g\rangle\,, \,\,\,\epsilon \in (0,1)\,.
  \end{align}
As $\|B\|\le 1/2$,  $\langle \cdot ,  \cdot \rangle_{\epsilon}$ is an inner product whose induced norm is equivalent to that of $\langle \cdot ,  \cdot \rangle$. 
As a consequence of Proposition \ref{thm:elem} one obtains for suitable
$f$ that satisfy $\mu_*[f]=0$:
\begin{align}\label{eq:decomp-hypo}  
\langle A f ,   f \rangle_\epsilon 
=& \langle S f ,  f \rangle + \epsilon 
\left[ 
\langle B S f,f \rangle  +  \langle  B T  f,f  \rangle  +  \langle S B f ,   f \rangle -  \langle T B f , f \rangle  
 \right]   \\
=&  \langle  (S - \epsilon TB) (I-\Pi)f ,  (I-\Pi)  f \rangle  + \epsilon  \langle B T \Pi  f ,  \Pi f \rangle \notag  \\
& + \epsilon\left[ \langle B S (I-\Pi) f ,  \Pi f \rangle  +  \langle B T (I-\Pi) f ,  \Pi f \rangle \right] \,, \notag
\end{align}
where we have used that $SB=0$.   The various terms in \eqref{eq:decomp-hypo} 
can be bounded as follows:
\begin{enumerate}
\item The term $\langle (S-\epsilon TB) (I-\Pi)f ,  (I-\Pi) f\rangle$ is controlled by the dissipative term in the $p$-variables (since $TB$ is bounded) and it is not difficult to see that in the cases considered here we have a Poincar\'e inequality in the $p$-variables (averaged over $\nu_*$):
\begin{align}\label{Poincare-p}
	 \langle f , -S  f \rangle \ge \lambda_p \|(I-\Pi) f\|^2
\end{align}
for some $\lambda_p>0$. For the Langevin equation $\lambda_p=\frac{\gamma}{\beta}$ is the spectral gap of the Ornstein-Uhlenbeck process, while for the other examples we can take $\lambda_p=\lambda$ from the velocity resampling mechanism.  
\item For the term $\langle B T \Pi  f ,  \Pi f \rangle$, note that using \eqref{eq:TPi} together with the Poincar\'e inequality for the measure $\nu_*$, we have 
 \begin{align}
	\langle   f , (T\Pi^* T\Pi) f \rangle =  \Pi\left(\frac{p^2}{m^2}\right) 
	 \|\nabla_q \Pi f\|^2 \ge \Pi\left(\frac{p^2}{m^2}\right) C_{\nu_*} \|\Pi f\|^2    \,,
\end{align}	
where $C_{\nu_*}$ is the Poincar\'e constant for the measure $\nu_*$. Then, since
 $ - BT\Pi = (I+(T\Pi)^* T\Pi)^{-1} (T\Pi)^* T\Pi$, by 
functional calculus we have  
\begin{align}
\langle -BT\Pi f, \Pi f \rangle \ge \left( 1 - \left(1 + \Pi\left(\frac{p^2}{m^2}\right) C_{\nu_*}\right)^{-1} \right)\|\Pi f \|^2 \equiv \lambda_q \|\Pi f \|^2\,,
\end{align}
where $\lambda_q\in(0,1)$.

\item For the off-diagonal terms it is enough to show that they are bounded, i.e., 
\begin{align} 
\|BT(I - \Pi)f\| + \|BS(I - \Pi)f\| \le R_0\|(I- \Pi)f\| \,.
\end{align}
The bound of the first term is the technical part of the proof; for the Langevin equation this is proved in \cite{Dolbeault2015}, and is generalized in \cite{ADNR2018} for the other samplers (see Lemma 29 and Lemma 32 in particular and the bound in Section 3.3 as well as the bound in Lemma 11 which is specific to the zig-zag sampler).

\end{enumerate}

Based on these estimates, one has constants $\lambda_q,\lambda_p,R_0> 0$ such that for any $f$ with $\mu_*[f]=0$:
\begin{align}\label{eq:matrix-hypo}  
\langle -A f ,  f \rangle_\epsilon \,&\ge \,\begin{bmatrix}
\|\Pi f\|\\
\|(I-\Pi)f\|
\end{bmatrix}^T\begin{bmatrix}
\epsilon\lambda_q &-\epsilon R_0/2\\
-\epsilon R_0/2 & \lambda_p-\epsilon
\end{bmatrix}\begin{bmatrix}
\|\Pi f\|\\
\|(I-\Pi)f\|
\end{bmatrix}\\
&\ge \Lambda(\epsilon) \Var_{\mu_*}[f] \notag \,,
\end{align}
where  
\begin{align}\label{Lambda_m_def}
\Lambda(\epsilon)\equiv\frac{(\lambda_q-1)\epsilon+\lambda_p -\sqrt{\left(( \lambda_q+1)\epsilon-\lambda_p \right)^2+\epsilon^2 R_0^2}}{2}
\end{align} is the smallest eigenvalue  of the matrix in \req{eq:matrix-hypo}. Note that $\Lambda(0)=0$ and $\Lambda^\prime(0)=\lambda_q>0$, therefore $\Lambda(\epsilon)>0$ for $\epsilon$ sufficiently small.  Specifically, $\Lambda(\epsilon)$ is positive if  
\begin{align}\label{eq:Lambda_pos}
0<\epsilon < 4 \lambda_q \lambda_p /( 4 \lambda_q + R_0^2).
\end{align}

In the next section, we  show how the Poincar{\'e} inequality (\ref{eq:matrix-hypo}) for the modified inner product (\ref{mod_norm_def}) can be used to derive non-asymptotic confidence intervals and UQ bounds for hypocoercive systems, having in mind the four examples outlined above.

\section{Concentration inequalities and UQ bounds  via Feynman-Kac semigroups}\label{sec:concentration}  
In this section, we prove our main new results for hypocoercive systems:
\begin{enumerate}
\item A concentration inequality and corresponding non-asymptotic confidence intervals in Section \ref{sec:conc_ineq}.
\item UQ bounds in Section \ref{sec:UQ}.
\end{enumerate}
  The former are obtained by an adaptation of the technique from \cite{Wu} and \cite{GaoGuillinWu} to hypocoercive systems, which we first summarize.

\subsection{Background}\label{sec:Kac_background}

As in \cite{Wu,GaoGuillinWu}, we will prove Bernstein-type concentration inequalities.  The following related  elementary facts will be used repeatedly (see, e.g., the discussion of sub-gamma random variables in Chapter 2 in \cite{Boucheron:2016}):

Consider the convex function $\Psi_{v,b}$ given by 
\begin{align} 
\Psi_{v,b}(\lambda) & = \frac{\lambda^2 v}{2(1-\lambda b)} \quad \textrm { for } 0 \le \lambda < 1/b \,. \label{eq:psi} 
\end{align}
Its (one-sided) Legendre transform  $\Psi_{v,b}^*$ is 
\begin{align}
\Psi_{v,b}^*(r) & =\sup_{0\le \lambda < 1/b}\left\{\lambda r - \Psi_{v,b}(\lambda)\right\} 
= \frac{2 r^2}{v \left(1 + \sqrt{1 + \frac{2br}{v} }\right)^2}
\quad \textrm {for } r \ge 0  \label{eq:psi*} 
\end{align}
and the inverse of the Legendre transform  $\Psi_{v,b}^*$  is 
\begin{align}
(\Psi_{v,b}^*)^{-1}(\eta) & = \inf_{\lambda>0} \left\{ \frac{\Psi_{v,b}(\lambda)+ \eta}{\lambda} \right\} \,=\,  \sqrt{2 v \eta } + b \eta \quad \textrm {for } \eta \ge 0  \,. \label{eq:psi*-1}  
\end{align}

%
Now we summarize the method of \cite{Wu,GaoGuillinWu}:\\
Let $\mathcal{X}$ be a Polish space  and suppose we have  time homogeneous, $\mathcal{X}$-valued, c\`adl\`ag Markov processes  $
(\Omega,\mathcal{F},\mathcal{F}_t,X_t,P^x)$, $x\in\mathcal{X}$, with initial distributions $(X_0)_*P^x=\delta_x$ for all $x$.  For an initial measure $\mu$, write $P^\mu = \int P^x  d\mu(x)$.

We assume that $\mu_*$ is an invariant ergodic measure on $\mathcal{X}$ 
consider the real Hilbert space $L^2(\mu_*)$  with scalar product 
$\langle \cdot, \cdot \rangle$.
Define  the strongly continuous Markov semigroup 
$T_t:L^2(\mu_*)\to L^2(\mu_*)$  by 
\begin{align}
	T_t[f](x)=E^x[f(X_t)]
\end{align}
and denote its generator   by $(A,D(A))$.   

More generally, for a bounded measurable $V:\mathcal{X}\to\mathbb{R}$, 
define the Feynman-Kac semigroup 
$T_t^V:L^2(\mu_*)\to L^2(\mu_*)$ by
\begin{align}
T_t^V[f](x)=E^x\left[f(X_t)e^{\int_0^t V(X_s)ds}\right],
\end{align}
which is a strongly continuous semigroup with generator $(A+V,D(A))$.  
If  we set
\begin{align}\label{kappa_def}
\kappa(V)\equiv\sup\left\{\langle (A+V)g,g\rangle:g\in D(A), \|g\| =1\right\}
\end{align} 
then, by definition (and as long as $\kappa(V)<\infty$), for any $g \in D(A)$  we have 
\begin{align}
\langle (A + V - \kappa(V)) g \,,\, g \rangle \le 0
\end{align} 
and thus by the Lumer-Philipps theorem (see, e.g., Chapter IX in \cite{YosidaFA})  the semigroup generated by   $A + V - \kappa(V)$ is a contraction semigroup on $L^2(\mu_*)$.  This implies that
\begin{align} \label{eq:Kac_semigroup_bound}
\| T_t^V \| \le  e^{t \kappa(V)}\,,\,\,\,t\geq 0
\end{align}
(note that \req{eq:Kac_semigroup_bound} also trivially holds if $\kappa(V)=\infty$). Therefore by the Chernoff bound we have 
\begin{align}\label{L2_concentration_ineq}
P^\mu\left(\frac{1}{T}\int_0^T f(X_t)dt-\mu_*[f]>r\right)
& \le  \inf_{\lambda >0} e^{-\lambda T r} E^\mu\left[e^{\lambda\int_0^T \widehat{ f}(X_t) dt}\right]   \notag \\
 &\le  \inf_{\lambda >0} e^{-\lambda T r}  \int  T_T^{\lambda\widehat f}[1] d\mu  \notag \\
& \le  \inf_{\lambda>0} e^{-\lambda T r}  \left\| \frac{d\mu}{d\mu_*}\right\| \left\| T_T^{\lambda\widehat f}\right\| \notag \\
&\leq \left\|\frac{d\mu}{d\mu_*}\right\| e^{-T\sup_{\lambda>0}\{ \lambda r- \kappa(\lambda\widehat{f})\}} \,.
\end{align}
This basic insight, first noted by \cite{Wu}, can also be extended to unbounded $V$.  From here, one can obtain explicit concentration inequalities 
by  further bounding $\kappa(\lambda\widehat{f})$ (which contains the Dirichlet form  $\langle A g \,,g\rangle$)  using  $L^2(\mu_*)$-functional inequalities, such as a Poincar{\'e} inequality  (or $\log$-Sobolev inequalities, Lyapunov functions, and so on...);  see \cite{Wu,lezaud:hal-00940906,cattiaux_guillin_2008,Guillin2009,GaoGuillinWu} for many such examples.  

%
%
%

\subsection{Concentration inequalities}\label{sec:conc_ineq}

In the hypocoercive  examples considered in this paper, the generator is non-reversible and there is no Poincar\'e inequality with 
respect to the $L^2(\mu_*)$-scalar product  but, as discussed in Section 2, there is  a Poincar\'e inequality in terms of a modified scalar product that induces an equivalent norm.  In the following theorem, we show that one still obtains concentration inequalities in this more general setting.

\begin{theorem}\label{thm:conc-star}{\bf (Concentration inequalities).} Let $
(\Omega,\mathcal{F},\mathcal{F}_t,X_t,P^x)$, $x\in\mathcal{X}$, be  $\mathcal{X}$-valued c\`adl\`ag Markov processes  with invariant ergodic measure  $\mu_*$. 

Let $\langle \cdot,\cdot\rangle_{\#}$ be an inner product on $L^2(\mu_*)$ such that 
\begin{enumerate}
\item The induced norms $\|\cdot\|_\#$ and $\| \cdot\|$ are equivalent:  there exists  
 $0<c\leq C<\infty$ such that $c\|\cdot\| \leq \|\cdot\|_\#\leq C\| \cdot \|$. 
\item For all $g\in L^2(\mu_*)$,  we have  $\langle g,1\rangle_\#=\langle g,1\rangle$. 
\item A Poincar\'e inequality holds for $\langle\cdot,\cdot\rangle_\#$, i.e.,  we have $\alpha>0$ such that
\begin{align}
 \|g\|_\#^2\leq \alpha\langle -Ag,g\rangle_\# \,\,\, \text{  for all $g\in D(A)$ with $\mu_*[g]=0$.}
 \end{align}
 \end{enumerate}
For bounded measurable $f$, let $M_{\widehat{f}}$ denote the multiplication operator 
by $\widehat{f}=f - \mu_*[f]$.  We have the following concentration inequalities for $T>0$:
\begin{align}\label{concentration_ineq_mod}
&P^\mu\left(\pm \left[ \frac{1}{T}\int_0^T f(X_t)dt-\mu_*[f]\right]\geq r\right)\leq c^{-1}\left\|\frac{d\mu}{d\mu_*}\right\| e^{-T \Psi^*_{v_\pm, b_\pm}(r)} \,,
\end{align}
where $\Psi^*_{\nu, b}$ is given in \eqref{eq:psi*}, 
 \begin{align}\label{eq:v+b+}
 v_\pm= 2 \alpha  \left\| \frac{1}{2}(M_{\pm \widehat{ f}}+M_{\pm \widehat{f}}^\dagger)1\right\|_\#^2\,,\,\,  \quad  b_\pm = \alpha  \max\left\{0, \sup_{\|g\|_\#=1}\langle M_{\pm \widehat{f}}g,g\rangle_\# \right\}\,,
 \end{align}
 and $M_{\widehat{f}}^\dagger$ denotes the adjoint with respect to the $\langle\cdot,\cdot\rangle_\#$-inner product.
%
%
\end{theorem}
Note that $\nu_\pm$ and $b_\pm$ can be replaced by any upper bounds on these quantities, for example in terms  of the $L^2(\mu_*)$-norm (see the calculation for the hypocoercive examples in Section \ref{sec:app} below).

\begin{proof}  The proof is a modification of  the  strategy used in \cite{GaoGuillinWu}.  We start as in  \req{L2_concentration_ineq} but use the Lumer-Phillips theorem \ for the  $\| \cdot \|_\#$ norm instead since, by equivalence of the norms, $T^{\lambda \widehat{f}}_t$ is also a strongly continuous semigroup on $(L^2(\mu_*),\|\cdot\|_\#)$ with the same generator.  Using the Chernoff bound, the equivalence of the norm and the fact that, by Assumption 2, $\|1\|_\# = \langle 1, 1 \rangle_\# =\langle 1, 1 \rangle =1$, we obtain
\begin{align}\label{Markov_ineq}
P^\mu\left(\frac{1}{T}\int_0^T f(X_t)dt \geq \mu_*[f]+r\right)
  & \leq \inf_{\lambda >0} e^{ -\lambda Tr} E^\mu\left[ e^{\lambda \int_0^T \widehat{f}(X_t)dt}\right] \notag \\
 & = \inf_{\lambda >0} e^{ -\lambda Tr} \int T^{\lambda \widehat{f}}_T[1]\frac{d\mu}{d\mu_*}d\mu_*\notag\\
& \leq \inf_{\lambda > 0} e^{ -\lambda Tr} \left\|\frac{d\mu}{d\mu_*}\right\|  \left\| T^{\lambda \widehat{f}}_T[1]\right\| \notag\\
& \leq \inf_{\lambda >0} e^{ -\lambda Tr}  \left\|\frac{d\mu}{d\mu_*}\right\| c^{-1} \left\|T^{\lambda \widehat{f}}_T\right\|_\# \|1\|_\#\notag \\
&  \leq c^{-1} \left\|\frac{d\mu}{d\mu_*}\right\| e^{-T \sup_{\lambda> 0} ( \lambda r - \kappa_{\#}(\lambda \widehat{f})}\,, 
\end{align}
where,  by the Lumer-Phillips theorem applied to $L^2(\mu_*)$ with the scalar product $\langle \cdot, \cdot \rangle_\#$,  
\begin{align}\label{kappa_star_def}
&\kappa_{\#}( \lambda \widehat{f} )\equiv\sup\left\{\langle (A+ \lambda \widehat{f})g,g\rangle_{\#} : g\in D(A),\|g\|_\#=1\right\}.
\end{align}
%
%
%
%
%
%
Next we use  the following lemma proved in \cite{BRBMarkovUQ}, which is a generalization of a result in 
\cite{GaoGuillinWu}, which itself was a simplification of the argument originally used in \cite{lezaud:hal-00940906}.  For completeness, the proof is given in the appendix.  

\begin{lemma}\label{perturb_lemma}
Let $H$ be  a real Hilbert space, $A:D(A)\subset H\to H$ a linear operator, and $M:H\to H$ a bounded linear operator; denote its adjoint by $M^\dagger$.  
Assume there exists $\alpha>0$ and $x_0\in H$ with $\|x_0\|=1$ such that
\begin{align}
\langle Mx_0,x_0\rangle=0\,\,\,\text{ and }\,\,\, \langle Ax,x\rangle \leq - \alpha^{-1} \|P^\perp x\|^2
\end{align}
for all $x\in D(A)$, where $P^\perp$ is the orthogonal projector onto $x_0^\perp$.
Then    
\begin{align}
\sup_{x\in D(A),\|x\|=1}  \langle (A+\lambda M)x,x\rangle \leq\frac{\lambda^2 \alpha V}{1- \lambda \alpha K} = \Psi_{2\alpha V, \alpha K}(\lambda) 
\end{align}
for $0\le \lambda < 1/\alpha K$, where 
\begin{align}
V = \left\| \frac{1}{2}(M+M^\dagger)x_0\right\|^2 \,, \quad K=\max\left\{0, \sup_{\|y\|=1} \langle My,y\rangle\right\}\,.
\end{align}
\end{lemma}

To use this result we take $H=(L^2(\mu_*),\langle\cdot,\cdot\rangle_\#)$ and let $A$  be the generator,  $M=M_{\widehat{f}}$,  and $x_0=1$. By Assumption 2, we have 
$\langle Mx_0\,,\,x_0\rangle_\#=\langle \widehat{f}\,, 1 \rangle_\# =\langle \widehat{f}\,, 1 \rangle =0$.  This assumption also implies that 
 the projection onto $1^\perp$ (for both  scalar products) is given by $P^\perp f = \widehat{f}$ and 
 \begin{align}
 \langle Ag,1\rangle_{\#}=\langle Ag,1\rangle=0,\,\,\, g\in D(A).
 \end{align}
  Combined with Assumption 3 and the fact that $A[1]=0$  we get 
\begin{align}
\langle Ag,g\rangle_\#=&\langle A\widehat{g},\widehat{g}\rangle_\# \le -\alpha^{-1}\|\widehat{g}\|^2_\#=-\alpha^{-1}\|P^\perp g\|^2_\#,\notag
\end{align}
and thus we can apply Lemma \ref{perturb_lemma} to obtain 
\begin{align} \label{eq:bernstein1}
\kappa_\#(\lambda\widehat{f})=\sup_{g\in D(A),\|g\|_\#=1}\langle (A+\lambda \widehat{f})g,g\rangle_\# \leq \Psi_{v_+,b_+}(\lambda) 
\end{align}
for all $0\leq \lambda< 1/b_+$,  where
\begin{align}
v_+=2 \alpha \left\| \frac{1}{2}(M_{\widehat{f}}+M_{\widehat{f}}^\dagger)1\right\|_\#^2\,,\,\,\,b_+ = \alpha  \max\left\{0, \sup_{\|g\|_\#=1}\langle M_{\widehat{f}}g,g\rangle_\# \right\}
\end{align}
 (as was given in \eqref{eq:v+b+}).  Therefore 
\begin{align} 
P^\mu\left(\frac{1}{T}\int_0^T f(X_t)dt \geq \mu_*[f]+r\right) & \le c^{-1} \left\|\frac{d\mu}{d\mu_*}\right\| e^{ - T \sup_{0\leq\lambda<1/b_+}\{ \lambda r - \Psi_{v_+,b_+}(\lambda)\}} \\
& = c^{-1} \left\|\frac{d\mu}{d\mu_*}\right\| e^{ - T \Psi^*_{v_+,b_+}(r)}  \,.\notag
\end{align}
The lower bound is obtained by replacing $f$ by $-f$ and this concludes the proof. 
%
%
%
%
%
%
\end{proof}

As an immediate corollary we obtain a  non-asymptotic confidence interval. 

\begin{corollary}\label{cor:confidence-interval}{\bf (Confidence intervals).}
Under the same assumptions as in Theorem \ref{thm:conc-star},
given a time $T$ and a confidence level $0 < 1-\delta < 1$ we have 
\begin{align}\label{eq:conf_interval}
P^\mu\left(\frac{1}{T}\int_0^Tf(X_t)dt-\mu_*[f]\in(-r_-,r_+)\right)\geq 1-\delta
\end{align}
where    
\begin{align}\label{eq:rpm}
r_\pm =\sqrt{2 v_\pm \frac{1}{T}\log\left(\frac{2N}{\delta}\right)} + b_\pm 
\frac{1}{T}\log\left(\frac{2N}{\delta}\right)\,,
\end{align}
%
%
with $N = c^{-1}\left\|\frac{d\mu}{d\mu_*}\right\|$ and $v_\pm$ and $b_\pm$ given in \eqref{eq:v+b+}.
\end{corollary}

\begin{proof} Define  $\eta= \frac{1}{T} \log\left( \frac{2N}{\delta}\right)$ (note that $N\geq 1$ follows from Assumptions 1 and 2 of Theorem \ref{thm:conc-star}), so that $r_\pm = (\Psi^*_{v_\pm,b_\pm})^{-1}(\eta)$, with $r_\pm$  given as in \eqref{eq:rpm}.  Using $r=r_\pm$ in the concentration bound in Theorem \ref{thm:conc-star} we  find 
\begin{align} 
P^\mu\left(\pm\left[\frac{1}{T}\int_0^Tf(X_t)dt-\mu_*[f]\right] \geq r_\pm\right) \le  N e^{- T \eta}=\frac{\delta}{2}.
\end{align}
The result (\ref{eq:conf_interval}) then follows from a union bound.
\end{proof}

%
%
 
\subsection{Robustness bounds on steady state bias due to model-form uncertainty}\label{sec:UQ}

Following on the  methods in \cite{GKRW,BRBMarkovUQ}, we can also use the above tools  to obtain bounds on the bias of the expectation of  
ergodic averages when the process itself is subject to (model-form) uncertainty.

We think of  the Markov process $(X_t,P^\mu)$ considered in Section 
\ref{sec:Kac_background} as the baseline process and consider  
an alternative stochastic process $(\widetilde X_t,\widetilde{P}^{\widetilde{\mu}})$ with initial distribution $(X_0)_*\widetilde{P}^{\widetilde{\mu}}=\widetilde{\mu}$ and let $\widetilde 
E^{\widetilde\mu}$ be the associated expectation.  
\begin{remark}
The requirements on the alternative process are very minimal.  In particular, we are {\em not} assuming $(\widetilde{X}_t,\widetilde P^{\widetilde\mu})$ is a Markov processes.
\end{remark}
 We will compare the two processes using relative entropy; we  assume absolute continuity of the path-space distributions on  finite time windows $[0,T]$, i.e., $\widetilde{P}^{\widetilde{\mu}}_T\ll P^\mu_T$, and also assume the relative entropy is finite:
 \begin{align}\label{eq:rel_ent_T}
R\big(\widetilde{P}^{\widetilde{\mu}}_T\|P^{\mu}_T\big) < \infty \,.
\end{align}
See the supplementary material to \cite{DKPP} for a collection of techniques that can be used to bound  the path-space relative entropy (\ref{eq:rel_ent_T}) for various classes of alternative models.

Given an observable $f$ we consider the ergodic averages 
 \begin{align}\label{eq:F_T}
 \widetilde{F}_T= \frac{1}{T}\int_0^T f(\widetilde{X}_t) dt \,,\,\,\,\, F_T= \frac{1}{T}\int_0^T f(X_t) dt\,,
\end{align}
and are interested in bounding the bias 
between the baseline and the alternative processes:
\begin{align}
\widetilde{E}^{\widetilde{\mu}}[\widetilde{F}_T] - E^{\mu}[F_T] \,.
\end{align}

%

\begin{theorem}\label{thm:UQ_mod_poincare}{\bf(Uncertainty Quantification bounds).}
Let $(X_t,P^x)$, $x\in\mathcal{X}$, be a family of Markov process satisfying the assumptions of Theorem \ref{thm:conc-star}, $\mu$ be an initial distribution, and  $(X_t,\widetilde{P}^{\widetilde{\mu}})$ be an alternative process with 
$R(\widetilde{P}^{\widetilde{\mu}}_T\|P^{\mu}_T) < \infty$.  Then for any bounded measurable $f$ we have
\begin{align}\label{ModPoincare_UQ_bound}
  \pm \left( \widetilde E^{\widetilde\mu} [ \widetilde{F}_T] - E^{\mu} [F_T ]\right)  \leq &  \sqrt{ 2 v_\pm \eta_T}+  b_\pm \eta_T   + \frac{C}{c}\frac{1-e^{-\alpha T}}{T} \left\|\frac{d \mu}{d\mu_*}\right\| 
{\rm Var}_{\mu_*}[f]\,, \notag
\end{align}
where $v_\pm$ and $b_\pm$ are given in \eqref{eq:v+b+} and 
\begin{align}
\eta_T =\frac{1}{T} \left(\log(c^{-1})+ \log \left\|\frac{d \mu}{d\mu_*}\right \| +  R\big(\widetilde P^{\widetilde{\mu}}_T\|P^{\mu}_T\big) \right)\,.\notag
\end{align}
If, in addition, the process $(\widetilde{X}_t,\widetilde{P}^{\widetilde{\mu}})$ is ergodic with invariant measure $\widetilde{\mu}_*$,  the limit
\begin{align}
\eta_\infty =  \lim_{T\to \infty} \frac{1}{T} R\big(\widetilde{P}^{\widetilde{\mu}}_T\|P^{\mu_*}_T\big)
\end{align}
exists for the relative entropy rate, and $\|d\mu/d\mu_*\|<\infty$, then we have the steady-state bias bound 
\begin{align}\label{ModPoincare_UQ_bound_infinity}
    \pm \left( {\widetilde \mu}_* [f]  -  \mu_* [f] \right)    \leq \sqrt{ 2 v_\pm  \eta_\infty}+ b_\pm \eta_\infty\,.
\end{align}
\end{theorem}

\begin{proof} The proof proceeds along the same line as in \cite{BRBMarkovUQ} to which we refer for more details.   
The starting point is the Gibbs information inequality 
\cite{chowdhary_dupuis_2013,DKPP}:  for $g$ bounded and measurable and  
probability measures  $Q$ and $\widetilde{Q}$
\begin{align}\label{goal_oriented_bound}
 \pm\left(  E_{\widetilde Q}E[g]-E_Q[g]\right) \leq  
\inf_{\lambda >0}  \left\{   \frac{\log E_Q[e^{ \pm \lambda 
(g - E_Q[g])}]  +R(\widetilde{Q} \|Q)}{\lambda} \right\} \,.
\end{align} 
This is a direct consequence of the Gibbs variational principle for the relative entropy, \cite{dupuis2011weak}. 

We apply the bound to the measures $P^{\mu}_T$,  $\widetilde P^\mu_T$ (distributions on path-space up to time $T$) and  $g(x)=\int_0^T f(x_t) dt$ (a bounded measurable function of paths, $x$, up to time $T$) and then divide both sides by $T$: 
\begin{align}
& \pm \left( {\widetilde E}^{\widetilde\mu} [ \widetilde{F}_T] -  E^{\mu} [F_T ] \right) 
\notag \\ 
&\,\,\, \leq \inf_{\lambda >0}  
     \left\{ \frac{ \log E^{\mu}[ e^{\pm \lambda T(F_T-E^\mu[F_T]}]  +R(\widetilde P^{\widetilde{\mu}}_T\|P^{\mu}_T) }{\lambda T} 
      \right\} 
\notag \\
&\,\,\, \leq 
      \underbrace{\inf_{\lambda >0}\left\{ 
      \frac{ \log E^{\mu}[ e^{\pm \lambda T(F_T-\mu_*[f]}]  +R(\widetilde P^{\widetilde{\mu}}_T\|P^{\mu}_T) }{\lambda T} \right\} }_{=\textrm{(I)}}
\mp \underbrace{ \left(E^\mu[F_T] - \mu_*[f]\right)}_{=\textrm{(II)}}\,.
  \end{align}    
The term (II) only involves the baseline process and is easily bounded, for example using the  Poincar\'e inequality for the scalar product $\langle\cdot,\cdot\rangle_\epsilon$:
\begin{align}\label{eq:UQB-II}
|\textrm{(II)}|= \left| E^{\mu}\left[ \frac{1}{T}\int_0^T \widehat{f}(X_t) \, dt\right]  \right| 
& \le  \frac{1}{T}\int_0^T \left|\left\langle \frac{d \mu}{d\mu_*},T_t[\widehat{f}]\right\rangle \right|dt  \\ 
& \le \frac{1}{T}\int_0^T e^{-t/\alpha}  \left\|\frac{d \mu}{d\mu_*}\right\|  \frac{C}{c}  \|\widehat{f}\|dt \notag \\
& =\frac{C}{c}\frac{1-e^{-T/\alpha}}{T/\alpha} \left\|\frac{d \mu}{d\mu_*}\right\| 
\sqrt{{\rm Var}_{\mu_*}[f]}\,. \notag
\end{align}
To bound the term (I),  we use Lemma \ref{perturb_lemma} to bound the moment generating function, similarly to the proof of Theorem \ref{thm:conc-star}:
\begin{align}\label{eq:UQB-I}
\textrm{(I)} &   = \inf_{\lambda >0} \left\{   
       \frac{\log \int T^{\pm \lambda \widehat{f}}_T[1]d\mu  +R\big(\widetilde P^{\widetilde{\mu}}_T\|P^{\mu}_T\big) }{\lambda T}
        \right\} \\
       &   \le  \inf_{\lambda >0} \left\{   
       \frac{\log  \left( c^{-1}\left\| \frac{d\mu}{d\mu_*}\right\|  e^{ T \kappa_\#(\pm \lambda \widehat{f})}  \right)       +R\big(\widetilde P^{\widetilde{\mu}}_T\|P^{\mu}_T\big) }{\lambda T}
        \right\} \notag \\
       &   =  \inf_{\lambda >0} \left\{ \frac{\kappa_\#(\pm \lambda \widehat{f}) + \eta_T}{\lambda} \right\} \notag \\
         &  \leq   \inf_{\lambda >0} 
       \left\{    \frac{ \Psi_{v_\pm,b_\pm}(\lambda) + \eta_T}{\lambda} \right\} \notag \\
       &= (\Psi_{v_\pm,b_\pm}^*)^{-1}(\eta_T) = \sqrt{2 v_\pm \eta_T} + b_\pm \eta_T \,. \notag
\end{align}
Finally, by taking  $T \to \infty$ we obtain  the  bounds in \req{ModPoincare_UQ_bound_infinity}
\end{proof}

\subsection{Application to hypocoercive samplers}\label{sec:app}

Theorems \ref{thm:hypo-conf} and \ref{thm:hypo-uq} 
for hypocoercive MCMC samplers follow rather immediately from Corollary \ref{cor:confidence-interval} and from Theorem \ref{thm:UQ_mod_poincare}.
We first verify the three assumptions in Theorem \ref{thm:conc-star}. 
The modified scalar product \eqref{mod_norm_def} has the form $\langle f , g \rangle_\epsilon = 
\langle f , g \rangle + \epsilon \langle f , G g \rangle$ 
where $G1=0$ and $\|G\|\le 1$. Therefore we  
we have  $c=(1-\epsilon)^{1/2}$, $C=(1+\epsilon)^{1/2}$, and 
$\langle f \,,\, 1 \rangle_\epsilon =  \langle f \,,\, 1 \rangle $,
and, for $\epsilon\in(0,1)$ sufficiently small (see \req{eq:Lambda_pos}), by \req{eq:matrix-hypo} we have $ \alpha=\frac{1+\epsilon}{\Lambda(\epsilon)}$.  

Since $\langle M_{\widehat{f}}g,g\rangle_\epsilon \le \|\widehat{f}\|_\infty \|g\|^2 (1+ \epsilon) \le \frac{1+\epsilon}{1-\epsilon} \|\widehat{f}\|_\infty \|g\|^2_\epsilon $ 
we have 
\begin{align}
b_\pm = \alpha \max\left\{0, \sup_{\|g\|_\epsilon=1}\langle M_{\pm\widehat{f}}g,g\rangle_\epsilon\right\}  \le \frac{(1+\epsilon)^2}{1-\epsilon}\frac{\|\widehat{f}\|_\infty}{\Lambda(\epsilon)}\,.
\end{align}    
Furthermore, using self-adjointness of $G$, we have 
\begin{align}
M_{\widehat{f}}^\dagger=& (I+ \epsilon G)^{-1}M_{\widehat{f}}(I+ \epsilon G) 
= M_{\widehat{f}} + \epsilon (I+ \epsilon G)^{-1} (M_{\widehat{f}}G -GM_{\widehat{f}}),
\end{align}
and thus, since $G1=0$,
\begin{align}
\frac{1}{2}(M_{\widehat{f}}+M_{\widehat{f}}^\dagger)1=\widehat{f}-\frac{\epsilon}{2} (I+\epsilon G )^{-1} G\widehat{f}\,. \notag
\end{align}
Therefore
\begin{align}
\left\|\frac{1}{2}(M_{\widehat{f}}+M_{\widehat{f}}^\dagger)1\right\|_\epsilon^2
&= 
\left\langle (I-\frac{\epsilon}{2} (I+\epsilon G )^{-1} G)\widehat{f} \,,\, (I+ \epsilon G)  (I-\frac{\epsilon}{2} (I+\epsilon G )^{-1} G) \widehat{f}\right\rangle \notag \\
&=\left\langle (I-\frac{\epsilon}{2} (I+\epsilon G )^{-1} G)\widehat{f} \,,\, 
(I+ \frac{\epsilon}{2} G) \widehat{f} \right\rangle \notag \\
&\le 
\left(1 + \frac{\epsilon}{2}\frac{1}{1 - \epsilon}\right)\left(1 + \frac{\epsilon}{2}\right)  
\|\widehat{f}\|^2 = \frac{1 - \frac{\epsilon^2}{4}}{1-\epsilon} {\rm Var}_{\mu_*}[f]\,,
\end{align}
and so 
\begin{align}
v_\pm \le \frac{ (1+\epsilon) (1 - \frac{\epsilon^2}{4})}{1-\epsilon} 
\frac{2 \Var_{\mu_*}[f]}{\Lambda(\epsilon)}\,.
\end{align}

\appendix  

\section{Additional Proofs}\label{app:additional_proofs}

For the sake of completeness and the convenience of the reader, in this appendix we provide   proofs of two results used above that have previously appeared elsewhere in the literature.

First we derive several   functional analytic estimates that form an important part of the hypocoercivity method of \cite{Dolbeault2009,Dolbeault2015}.
\begin{proof}[Proof of Proposition \ref{thm:elem}]
The first property follows from $A 1=A^*1=0$ and $\Pi1=1$.

For (2),  it easy to verify that  $S\Pi =0$ and taking adjoint gives $\Pi S=0$.   

For (3), note that  $T \Pi f = v \nabla_q \Pi f$ and thus 
$ \Pi T \Pi f =  \Pi (v  \nabla_q \Pi f) =  (\nabla_q \Pi f ) \Pi v =0$ (since the velocity $v$ has mean zero).  

For (4), note that  by (3) we have $\Pi T \Pi =0$ and thus $B \Pi  =0$.  On the other hand, by definition of $B$ we have the identity 
\begin{align}\label{eq:id01}
Bf +  (T \Pi)^* (T \Pi) Bf =  \Pi T f \,,
\end{align}
and thus $\Pi B = B$. 

Taking the scalar product of \req{eq:id01} with $Bf$ and using $\Pi B=B$ and $ T \Pi = (I - \Pi) T \Pi$ we obtain 
\begin{align}
  \langle Bf\,,\, Bf\rangle + \langle T Bf, TBf \rangle & = \langle - T B f  \,,\, (I - \Pi) f \rangle  \\
  & \le \| ( I- \Pi)f \| \| TB f\| \notag\\
  & \le \frac{1}{4}\| ( I- \Pi)f \|^2 +    \| TB f\|^2 \,.\notag
\end{align}
The last inequality gives $\|Bf \| \le \frac{1}{2} \| ( I- \Pi)f \|$ while the first inequality gives $\|TB f\| \le \| ( I- \Pi)f \|$.
\end{proof}

We end with a derivation of bounds on perturbations to the generator (previously obtained in \cite{BRBMarkovUQ}) that play a key role in proving our new results in Theorems \ref{thm:conc-star} and \ref{thm:UQ_mod_poincare}.
\begin{proof}[Proof of Lemma \ref{perturb_lemma}]
Let $x\in D(A)$ with $\|x\|=1$.  Define $a =\langle x_0,x\rangle$ so that    
$\|P^\perp x\|^2=1-|a |^2$ and $|a |\leq 1$ with equality if and only if $P^\perp x=0$. We can decompose $x=a  x_0+\sqrt{1-|a |^2}v$, where: (a) $P^\perp x=0$, $|a|=1$, and $v=0$ or (b) $P^\perp x\neq 0$, $v=P^\perp x/\sqrt{1-|a|^2}$, and $\|v\|=1$.   In either case, $v\perp x_0$. 

Using 
$\langle Mx_0,x_0\rangle =0$ and $\langle A x , x\rangle \le -  \alpha^{-1} 
\|P^\perp x\|$ 
one obtains
\begin{align}
& \langle (A+ \lambda M)x,x\rangle \notag \\ 
& \leq   - \alpha^{-1} (1-|a |^2) +2\lambda a \sqrt{1-|a |^2}  \langle v, \frac{1}{2} (M + M^\dagger) x_0\rangle +\lambda (1-|a|^2) \langle Mv,v\rangle\notag\\
 & \leq   2\lambda |a| \sqrt{1-|a|^2} V^{1/2} -(1-|a |^2) \left( \alpha^{-1}-\lambda K\right)\,, \notag
\end{align}	
where $V=\left\| \frac{1}{2} (M + M^\dagger) x_0\right  \|^2$ and $K = \max\left\{0, \sup_{\|v\|=1}\langle Mv,v\rangle\right\}$.  
Restricting to $0\leq \lambda  <1/\alpha K$ and using $|a|\le 1$  we can estimate 
\begin{align}
\sup_{x\in D(A),\|x\|=1} \langle (A+\lambda M)x,x\rangle \leq&\sup_{r\geq 0}\left(2\lambda V^{1/2} r-\left(\alpha^{-1}-\lambda K \right)r^2\right)=\frac{\lambda^2 \alpha V }{1-\lambda \alpha K}\,. \notag 
\end{align}
\end{proof}

\subsection*{Acknowledgments}
Research supported in part by the National Science Foundation (DMS-1515712, DMS-2008970) and the Air Force Office of Scientific Research (AFOSR) (FA-9550-18-1-0214). Luc Rey-Bellet thanks Gabriel Stoltz and Stefano Olla for useful discussions and suggestions.

\bibliographystyle{amsplain}
\bibliography{conc_UQ_hypocoercive.bbl}
\end{document}